\documentclass[review]{elsarticle}

\usepackage{lineno,hyperref}
\usepackage{amsmath}
\usepackage{amsthm}
\usepackage{amssymb}
\usepackage{amsfonts}
\newtheorem{theorem}{Theorem}

\newtheorem{corollary}[theorem]{Corollary}

\newtheorem{remark}[theorem]{Remark}

\journal{Journal of \LaTeX\ Templates}

\begin{document}

\begin{frontmatter}

\title{A sufficient condition for local nonnegativity\tnoteref{mytitlenote}}
\tnotetext[mytitlenote]{This work was partially supported by National Nature Science Foundation (61673016), the Fundamental Research Funds for the Central Universities, SWUN (2017NZYQN42), Innovative Research Team of the Education Department of Sichuan Province (15TD0050), Sichuan Youth Science and Technology Innovation Research Team (2017TD0028).}

\author[jx]{Jia Xu\corref{cor1}}
\ead{xufine@163.com}

\author[yy]{Yong Yao}
\ead{yaoyong@casit.ac.cn}

\cortext[cor1]{Corresponding author.}

\address[jx]{School of Computer Science and Technology,
Southwest Minzu University, Chengdu, Sichuan 610041,
China}

\address[yy]{Chengdu Institute of Computer Applications, Chinese
Academy of Sciences, Chengdu, Sichuan 610041, China}

\begin{abstract}
A real polynomial $f$ is called local nonnegative at a point $p$, if it is nonnegative in a
neighbourhood of $p$. In this paper, a sufficient condition for determining this property is constructed.
Newton's principal part of $f$ (denoted as $f_N$) plays a key role in this process. We proved that
if every $F$-face, $(f_N)_F$, of $f_N$ is strictly positive over $(\mathbb{R}\setminus 0)^n$, then
$f$ is local nonnegative at the origin $O$.
\end{abstract}


\end{frontmatter}


\section{Introduction}

Before begin, several notations will be given.
Let $\mathbb{N},\ \mathbb{R},\ \mathbb{R}_{>0},\ \mathbb{R}_{\geq 0}$ denote the set of all natural numbers,
real numbers, positive real numbers and nonnegative real numbers, respectively. Let us start with a question.

What is local nonnegative?  Here is an explicit definition.

Let $f\in \mathbb{R}[x_1,x_2,\cdots,x_n]$ be a real polynomial.
Consider the point $p=(p_1,\cdots,p_n)\in \mathbb{R}^n$. If there is a $p$'s neighbourhood
$B(\epsilon,p)=\{x|\ \parallel x-p \parallel < \epsilon \}$ such that
$$\forall x\in B(\epsilon,p),\  f(x)\geq 0,$$
then $f$ is called local nonnegative at $p$.

There is a variety of  mathematical problems related to local nonnegativity. For instance, the problem of determining
the local minimum of a polynomial, finding the isolated singular point of a real algebraic variety, constructing Lyapunov
functions, and proving algebra inequalities.

So, what kind of function is local nonnegative at a point?

This problem will be discussed in two cases. Pick a point $p\in \mathbb{R}^n$, and let $f$ be a real function.
Case 1: $f(p)>0$. It is a trivial situation. Since a familiar result said that
if $f$ is continuous in an open set $D\in \mathbb{R}^n$,
and there is a point $p\in D$ such that $f(p)>0$, then, there exists a neighbourhood of $p$ such that $f$ is still positive
in the neighbourhood. Thus, it is obvious that $f$ is local nonnegative at the point $p$.
Case 2: $f(p)=0$. It becomes a tricky problem.
To facilitate discussion, we will assume $f$ to be a polynomial, and denote $O=(0,\ldots,0)$ as the origin.
 By coordinate transformation, the problem that whether $f$ is local nonnegative
at an arbitrary point can be reduced to determining whether it is local nonnegative at the origin.
Hence, throughout the paper, the local nonnegativity of $f$ means that it is local nonnegative at the origin.
Then, there is a renowned result about Hessian matrix in the standard textbooks (\cite{Apo},\cite{Zor},\cite{CS}),  which may
come in handy.

\newtheorem{th:1}{Theorem}
\begin{th:1}\upshape\label{th:1}
Let $\ D\subseteq \mathbb{R}^n$ be an open set. Assume that all second-order partial derivatives of
 $f: D\rightarrow\mathbb{R}$ exist and are continuous at the origin $O$ , and assume further that
 $$f(O)=0,\ \frac{\partial f}{\partial x_i}(O)=0, \ 1\leq i \leq n.$$
 If Hessian matrix
 $$ H_f(O)=\left (\frac{\partial^2 f}{\partial x_i \partial x_j}(O)  \right )_{1\leq i,j\leq n} $$
 is positive definite, then $f$ is local nonnegative (and $f$ has no other zeros neighbouring $O$).
\end{th:1}

Unfortunately, this law will failed if Hessian matrix is not positive definite. Here is an example that Hessian matrix is a singular matrix
(i.e., the determinant of Hessian matrix is $0$ ), whereas the polynomial is local nonnegative.

\newtheorem{ex:1}{Example}
\begin{ex:1}\upshape\label{ex:1}
For an arbitrary real number $s$, the polynomial
$$f=x^2+y^4+z^6-sxy^2z^3,\ f\in \mathbb{R}[x,y,z] .$$
is local nonnegative.
\end{ex:1}

Roughly speaking, if Hessian matrix is singular, then it is tough to determine local nonnegativity.

However, it is easy to find a necessary condition of local nonnegativity:

Let $\varphi_1(t),\allowbreak  \cdots, \varphi_n(t)\in \mathbb{R}[t]$ be $n$ arbitrary polynomials in a variable.
Then the polynomial $f$ is local nonnegative only if,
$f(\varphi_1(t), \cdots, \varphi_n(t))\equiv 0$, or
the degree of the leading term (i.e. the term with lowest degree ) of $f(\varphi_1(t), \cdots, \varphi_n(t))$ is even
and its leading coefficient (i.e. the coefficient of leading term ) is positive.

This result can be used to prove that a polynomial is not local nonnegative. For instance, if there exist
polynomials $\varphi_1(t),\allowbreak  \cdots, \varphi_n(t)\in \mathbb{R}[t]$,
such that the leading coefficient of~$f(\varphi_1(t), \cdots, \varphi_n(t))$ is negative,
then $f$ is not local nonnegative. Here is an example.

\newtheorem{ex:2}[ex:1]{Example}
\begin{ex:2}\upshape\label{ex:2}
Consider the polynomial
$$f=x^{16}+y^{18}-x^7y^3+x^{12}y^{15}+x^4y^2-2x^3y^3+y^4x^2.$$
Pick $x=t,\ y=t$, the above polynomial can written as
$$f(t,t)=t^{27}+t^{18}+t^{16}-t^{10}=t^{10}(-1+t^6+t^8+t^{17}).$$
It is obvious that the coefficient of leading term is $-1$. Then, $f(t,t)<0$ when $t\rightarrow 0$.
Thus, $f$ is not local nonnegative.
\end{ex:2}

In contrast, it is a hard problem to obtain a sufficient condition of local nonnegativity.
We hardly get any other result about the sufficient condition except the Theorem \ref{th:1}.
The aim of this paper is to construct an available sufficient condition for determining the local nonnegativity.

\section{Preliminaries}

Let $x^{\alpha}=x_1^{\alpha_1}\cdots x_n^{\alpha_n}$ with $\alpha=(\alpha_1,\ldots,\alpha_n)\in
\mathbb{N}^n$. Then, a polynomial $f\in \mathbb{R}[x_1,\cdots,x_n]$  can be written as $$f=\sum_{\alpha}a_{\alpha}x^{\alpha}. $$
The support of $f$ is defined as $\mathrm{Sup}(f)=\{\alpha\in \mathbb{N}^n|\ a_{\alpha}\neq 0 \}$, and let
$$\mathrm{Sup}^+(f)=\{\alpha\in \mathbb{N}^n|\ a_{\alpha}> 0 \}, \quad \mathrm{Sup}^-(f)=\{\alpha\in \mathbb{N}^n|\ a_{\alpha}< 0 \}.$$
Here $\mathrm{Sup}^+(f)$ and $\mathrm{Sup}^-(f)$ stand for the positive and negative support of $f$ respectively.
Denote
$$f^+=\sum_{\alpha\in \mathrm{Sup}^+(f)}a_{\alpha}x^{\alpha},\
f^-=\sum_{\alpha\in \mathrm{Sup}^-(f)}a_{\alpha}x^{\alpha}.$$

Let $N(f)$ be the Newton polytope of $f$. It is the convex hull of $\mathrm{Sup}(f)$, that is,
$$N(f)=\mathrm{Conv}\{ \alpha\in \mathrm{Sup}(f)\}.$$
Let $\mathcal{F}(f)$ be a set of all the nonempty faces of $N(f)$. For $F\in \mathcal{F}(f)$, the $F$-face of the polynomial $f$  is
$$f_F=\sum_{{\alpha}\in (F\bigcap \mathrm{Sup}(f))} a_{\alpha}x^{\alpha}.$$
Noticed that $f_F=f$ when $F=N(f)$.

\newtheorem{def:b10}{Definition}[section]
\begin{def:b10}\upshape\label{def:b10}
Let $V$ denote the set of vertices of $N(f)$. Define
$$f^V=\sum_{\upsilon\in V} x^{\upsilon}.$$
$f^V$ is called the characteristic polynomial of vertices of $f$.
\end{def:b10}

To study the local nonnegativity of polynomials, we need to present the definition of Newton's principal part.
For a polynomial $f$, define set of points
\begin{equation}\label{eq:b10}
  N_f=\mathrm{Conv}(\bigcup_{\alpha\in \mathrm{Sup}(f)} (\alpha + \mathbb{R}^n_{\geq 0})).
\end{equation}
The compact face of $N_f$ is denoted as $\Gamma_f$, which is called the Newton's diagram of $f$.
Then, Newton's principal part is defined as follows.

\newtheorem{def:b20}[def:b10]{Definition}
\begin{def:b20}\upshape\label{def:b20}
The polynomial
$$f_N=\sum_{{\alpha}\in (\Gamma_f\bigcap \mathrm{Sup}(f))} a_{\alpha}x^{\alpha},$$
is called Newton's principal part of $f$.
\end{def:b20}

Thus, an arbitrary polynomial $f$ can be written as
$$f=f_{N}+f_{\overline{N}},$$
where $f_{N}$ is the Newton's principal part of $f$, and $f_{\overline{N}}$ is the remainder (it may be zero polynomial).
Especially for the polynomial in one variable, $f_N$ is the term of $f$ with lowest degree.

The definition of Newton's principal part of a polynomial is a key to this paper.
It is easy to see that for a polynomial $f$, the condition that Hessian matrix
is positive definite in Theorem 1 is exactly the requirement that the Newton's principal part $f_N$ is
a positive definite quadratic form. Hence, a normal idea is to generalize this result and hope
that a polynomial is local nonnegative based on the premise that its Newton's principal part is strictly positive
on $\mathbb{R}^n\setminus O$. Unfortunately, this conjecture is proved wrong. The counterexample is Example \ref{ex:2}. Given
$$f=x^{16}+y^{18}-x^7y^3+x^{12}y^{15}+x^4y^2-2x^3y^3+y^4x^2,$$
its Newton's principal part is
\begin{eqnarray*}
f_N&=&x^{16}+x^4y^2-2x^3y^3+y^4x^2+y^{18}\\
&=&x^{16}+y^{18}+(x^2y-y^2x)^2.
\end{eqnarray*}
$f_N$ is strictly positive on $\mathbb{R}^2\setminus O$, whereas $f$ is not local nonnegative. Moreover,
we noticed that one $F$-face,  $(x^2y-y^2x)^2$, of $f_N$ is not strictly positive definite on $\mathbb{R}^2\setminus O$.
Because it vanishes when $x=y$. The Newton's diagram of $f$ is a polygonal line in Fig. 1.

\vspace{0.5cm}
 \setlength{\unitlength}{0.1in}
\begin{picture}(20,20)
\put(12,-2){\makebox(15,2)[1] {O}}
\put(20,0){\vector(0,1){20}}
\put(20,0){\vector(1,0){20}}
\put(20,18){\circle*{0.5}   } \put(36,0){\circle*{0.5} }
\put(24,2){\circle*{0.5} }  \put(22.2,4){\circle*{0.5} }
\put(20,18){\line(1,-6){2.3}}
\put(22,4){\line(1,-1){2.2}}
\put(24,2){\line(6,-1){12}}
\put(23,3){\circle*{0.5} }
\put(27,3){\circle*{0.5} }
\put(32,15){\circle*{0.5} }
 \put(20,-4.5){\makebox(15,2)[1]{\hbox{\small Fig.1 . $f$ 的~Newton 图~$\Gamma_f$}}}
\end{picture}
\vspace{1.5 cm}

It is time to present the main result of this paper.

\newtheorem{th:2}[th:1]{Theorem}
\begin{th:2}\upshape\label{th:2}
  Let $f_N$ be the Newton's principal part of a polynomial $f\in \mathbb{R}[x_1,\cdots,x_n]$ (see Definition \ref{def:b20}).
  If every $F$-face, $(f_N)_F$, of $f_N$ satisfies
$$\forall x\in (\mathbb{R}\setminus 0)^n,\ (f_{N})_F> 0,$$ then $f$ is local nonnegative.
\end{th:2}

\section{Proof of the Main Result}

In this section, we will present the proof of the main result (Theorem \ref{th:2}). Firstly, several lemmas are developed.

\newtheorem{lem:c10}{Lemma}[section]
\begin{lem:c10}\upshape\label{lem:c10}
Given $\alpha^{(1)},\cdots,\alpha^{(m)}  \in \mathbb{N}^n$, $\beta$ can be written as a convex combination of
\ $\alpha^{(1)},\cdots,\alpha^{(m)}$, i.e., there exists nonnegative real numbers $\lambda_1,\cdots,\lambda_m$ satisfying
$$\beta=\lambda_1 \alpha^{(1)}+\cdots+\lambda_m \alpha^{(m)}, \ \lambda_1 +\cdots+\lambda_m =1.$$
Then, there exists a positive constant $k_\beta$ such that
$$\forall x\in  \mathbb{R}_{\geq 0}^n,\ x^{\alpha^{(1)}}+\cdots+x^{\alpha^{(m)}}\geq k_\beta x^{\beta}.$$
\end{lem:c10}

\begin{proof}
By generalized mean inequality \cite{HLP},
$$\forall x\in  \mathbb{R}_{\geq 0}^n,\ \lambda_1 x^{\alpha^{(1)}}+\cdots+\lambda_m x^{\alpha^{(m)}}\geq
 x^{\lambda_1 \alpha^{(1)}+\cdots+\lambda_m \alpha^{(m)}}=x^{\beta}.$$

 Let $k_\beta=1/\max \{\lambda_1, \cdots, \lambda_m \}$ . Then,
$$x^{\alpha^{(1)}}+\cdots+ x^{\alpha^{(m)}} \geq k_\beta x^{\beta}$$ holds.
\end{proof}

\newtheorem{lem:c20}[lem:c10]{Lemma}
\begin{lem:c20}[Handelman\cite{Han,AT}]\label{lem:c20}
Let $h\in \mathbb{R}[x_1,\cdots,x_n]$ be a nonzero polynomial. Then, there exists a positive integer $m$ such that
 all the coefficients of $h(h^{+}-h^{-})^m$ are nonnegative real numbers if and only if,
 $$\forall F\in \mathfrak{\mathcal{F}}(h),\ \forall x\in \mathbb{R}_{>0}^n,\  h_{F}> 0 . $$
\end{lem:c20}

\begin{remark}
All of the coefficients of $(h^{+}-h^{-})$ are nonnegative real numbers.
\end{remark}

\newtheorem{lem:c30}[lem:c10]{Lemma}
\begin{lem:c30}\label{lem:c30}
Let $h\in \mathbb{R}[x_1,\cdots,x_n]$ be a nonzero polynomial. Then, there exists a constant $\tau >0$ satisfying
$$\forall x\in \mathbb{R}_{>0}^n,\ h\geq \tau h^V$$
if and only if,
$$\forall F\in \mathfrak{\mathcal{F}}(h),\ \forall x\in \mathbb{R}_{>0}^n,\  h_{F}> 0, $$
where $h^V$ is a characteristic polynomial of vertices of $h$.
\end{lem:c30}

\begin{proof}
$\Leftarrow$: Let $V=\{v^{(1)},v^{(2)},\cdots,v^{(d)}\}$ be the set of vertices of $N(h)$, where $N(h)$ is the Newton polytope of $h$.
Then, the characteristic polynomial of vertices of $h$ is
$$
h^V=\sum_{i=1}^d x^{v(i)}.
$$
Moreover, $V$ is also the set of vertices of $N(h^+-h^-)$ because $(h^+-h^-)$ and $h$ have the same Newton polytope.
For the polynomial $p$ and $q$, the Newton polytope of their multiplication corresponds to the Minkowski sum of their polytopes[7,\ 8], that is,
$$N(pq)=N(p)+N(q).$$
Hence,
$$N(h(h^{+}-h^{-})^m)=\underbrace{N(h)+N(h)+\cdots +N(h)}_{m+1}=N((h^{+}-h^{-})^{m+1}).$$
Note that $(m+1)v^{(1)},\cdots, (m+1)v^{(d)}$ are vertices of both $N((h^{+}-h^{-})^{m+1})$ and
$N(h(h^{+}-h^{-})^m)$. Thus we have
$$(m+1)v^{(1)},\cdots, (m+1)v^{(d)}\in \mathrm{Sup}(h(h^{+}-h^{-})^m).$$

Applying Handelman's theorem (Lemma \ref{lem:c20}) to $h$, then there exists a positive integer $m$ such that
all coefficients of $h(h^{+}-h^{-})^m$ are nonnegative real numbers. There exists a positive constant $k_1$, therefore, satisfying
\begin{equation}\label{eq:c10}
\forall x\in \mathbb{R}_{>0}^n,\ h(h^{+}-h^{-})^m \geq k_1 \sum_{i=1}^d  x^{(m+1)v^{(i)}}.
\end{equation}
By H\"{o}lder's inequality, it holds that
\begin{equation}\label{eq:c20}
d^m\sum_{i=1}^d  x^{(m+1)v^{(i)}}\geq (\sum_{i=1}^d  x^{v^{(i)}})^{m+1}=(h^V)^{m+1}.
\end{equation}
\eqref{eq:c10}, together with \eqref{eq:c20} indicate that
\begin{equation}\label{eq:c30}
h\geq \frac{k_1(h^V)^{m+1}}{d^m (h^{+}-h^{-})^m}.
\end{equation}

It is well known that all the points on a convex polyhedron can be presented by convex combinations of the elements in the set of its vertices.
Hence, each term of $(h^{+}-h^{-})$ can leads to an inequality by using Lemma \ref{lem:c10}.
Add these inequalities, and then, there exists a positive constant $k$ satisfying
\begin{equation}\label{eq:c40}
\forall x\in \mathbb{R}_{>0}^n,\ (h^{+}-h^{-})\leq k h^V(x).
\end{equation}

Pick $\tau=k_1/ (k d)^m$.  Combining \eqref{eq:c30} and \eqref{eq:c40} imply that
$$h\geq \tau h^V.$$

$\Rightarrow$: Conversely, assume that there exists a $F$-face~$h_F$ and a point~$x'\in R^n_{>0}$ such that
~$h_F(x')\leq 0$. On one hand, by~Definition \ref{def:b10}, it is clearly,
$$x'\in R^n_{>0},\ (h_F)^V(x')>0.$$
On the other hand, choose a vector $\nu=(\nu_1,\cdots,\nu_n)\in \mathbb{R}^n$ such that the dot product
$\alpha \cdot \nu\ (\alpha\in \mathrm{Sup}(h))$ arrives at the maximum on the face~$F$,
which is denoted as\ $M$. Note that the dot product ~$\alpha \cdot \nu$ arrives at the same $M$  on~$\mathrm{Sup}((h_F)^V)$.

Let~$x'=(x'_1,\ldots,x'_n)$. Then,
$$h(x'_1 e^{\nu_{1}t},\ldots,x'_n e^{\nu_{n}t})=
\sum_{\alpha\in \mathrm{Sup}(h)} a_\alpha (x')^\alpha e^{(\alpha \cdot \nu)t},$$
and
$$h^V(x'_1 e^{\nu_{1}t},\ldots,x'_n e^{\nu_{n}t} )=
\sum_{\alpha\in V} a_\alpha (x')^\alpha e^{(\alpha \cdot \nu)t}. $$

Hence, we have
\begin{equation}\label{eq:c50}
\lim_{t\rightarrow \infty} e^{-Mt}h(x'_1 e^{\nu_{1}t},\ldots,x'_n e^{\nu_{n}t} )=h_F(x'),
\end{equation}
and
\begin{equation}\label{eq:c60}
\lim_{t\rightarrow \infty} e^{-Mt}h^V(x'_1 e^{\nu_{1}t},\ldots,x'_n e^{\nu_{n}t} )=(h_F)^V(x').
\end{equation}

If there is a positive constant~$\tau$ satisfying
\begin{equation}\label{eq:c70}
\forall x\in \mathbb{R}_{>0}^n,\ h\geq \tau h^V.
\end{equation}
Then,~\eqref{eq:c50} together with\eqref{eq:c60} and \eqref{eq:c70}, indicate that
\begin{equation}
h_F(x')\geq \tau (h_F)^V(x'),
\end{equation}
contradicting the assumption that~$h_F(x')\leq 0$ while $\ (h_F)^V(x')>0$.
\end{proof}

Before proving the main result, we continue discuss the polynomial of Example 1, which will illuminate the proof of
Theorem \ref{th:2}.

Consider
$$g=x^2+y^4+z^6-sxy^2z^3,$$
where $s$ is a given real number. If write $g$ as a polynomial in $x$, then its lowest degree is $1$ with coefficient $-sy^2z^3$.
Namely, the coefficient of the term in $x$ with lowest degree is indefinite. However, $g$ is still local nonnegative.
In fact, by inequality of arithmetic and geometric mean, it holds that
$$x^2+y^4+z^6-3 x^{\frac{2}{3}}y^{\frac{4}{3}}z^2\geq 0.$$
Thus $g$ can be written as
$$g=(x^2+y^4+z^6-3 x^{\frac{2}{3}}y^{\frac{4}{3}}z^2)+
(x^{\frac{1}{3}}y^{\frac{2}{3}}z)^2(3-sx^{\frac{1}{3}}y^{\frac{2}{3}}z).$$
Noticed that
$$\lim_{
(x,y,z)\rightarrow (0,0,0)} sx^{\frac{1}{3}}y^{\frac{2}{3}}z=0.$$
Hence, $g$ is local nonnegative.

\newtheorem*{thm}{Theorem 2}
\begin{thm}\upshape
  Let $f_N$ be the Newton's principal part of a polynomial $f\in \mathbb{R}[x_1,\cdots,x_n]$ (see Definition \ref{def:b20}).
  If every $F$-face, $(f_N)_F$, of $f_N$ satisfies
$$\forall x\in (\mathbb{R}\setminus 0)^n,\ (f_{N})_F> 0,$$ then $f$ is local nonnegative.
\end{thm}

\begin{proof}
The result obviously holds if $f=f_N$ since $f$ is global nonnegative, and so, in particular it is local nonnegative.

Next, if $f\neq f_N$, then $f$ can be written as
\begin{equation}\label{eq:c80}
\begin{split}
f=& f_N+f_{\overline{N}}\\
=&f_N+\sum_{\beta\in \mathrm{Sup}(f_{\overline{N}})} a_{\beta}x^{\beta}.
\end{split}
\end{equation}
Let T be the number of elements in the set $\mathrm{Sup}(f_{\overline{N}})$, i.e. $T=|\mathrm{Sup}(f_{\overline{N}})|$.
By hypothesis, $f_N$ is nonnegative over $\mathbb{R}^n$, this means that $f(a_1,\ldots,a_n)\geq 0$
for all $(a_1,\ldots,a_n)\in \mathbb{R}^n$, and, hence, components of each vertex must be even ( possibly be zero ).
Thus, it reduces to discussing whether the result holds in $\mathbb{R}_{\geq 0}^n$.
Consequently, assume that $x \in \mathbb{R}_{\geq 0}^n$.

Firstly, by Lemma \ref{lem:c30}, there exists a positive constant $\tau$ such that
\begin{equation}\label{eq:c90}
f_N-\tau (f_N)^V\geq 0.
\end{equation}

Secondly, consider every term $a_{\beta} x^{\beta}$ of $f_{\overline{N}}$. By \eqref{eq:b10},
there exists a monomial $x^{\hat{\beta}}$ such that
$$\hat{\beta}\in N(f_N) , \ \beta=\hat{\beta} +\delta,\ \delta \in \mathbb{R}^n_{\geq 0}.$$
Note that $N(f_N)=\mathrm{Conv}\{ \alpha\in \mathrm{Sup}(f_N)\}.$ By Lemma \ref{lem:c10}, there exists a positive
constant $k_{\hat{\beta}}$ satisfying
\begin{equation}\label{eq:c95}
(f_N)^V\geq k_{\hat{\beta}} x^{\hat{\beta}}.
\end{equation}

Moreover,
\begin{equation}\label{eq:c100}
\begin{split}
\tau (f_N)^V+f_{\overline{N}}=& \tau (f_N)^V+\sum_{\beta\in \mathrm{Sup}(f_{\overline{N}})}a_{\beta} x^{\beta}\\
=&\sum_{\beta\in \mathrm{Sup}(f_{\overline{N}})} \left(\frac{\tau}{T}(f_N)^V+a_{\beta} x^{\beta}\right)\\
=&\sum_{\beta\in \mathrm{Sup}(f_{\overline{N}})}\left ( \frac{\tau}{T}
\left((f_N)^V-k_{\hat{\beta}} x^{\hat{\beta}} \right)+
 \left (\frac{\tau k_{\hat{\beta}}}{T} x^{\hat{\beta}}+a_{\beta} x^{\beta}\right )\right ) \\
 =&\sum_{\beta\in \mathrm{Sup}(f_{\overline{N}})} \frac{\tau}{T}
\left((f_N)^V-k_{\hat{\beta}} x^{\hat{\beta}} \right)+
\sum_{\beta\in \mathrm{Sup}(f_{\overline{N}})}  x^{\hat{\beta}}
\left (\frac{\tau k_{\hat{\beta}}}{T}+a_{\beta} x^{\delta}\right ).
\end{split}
\end{equation}

This, together with \eqref{eq:c80} and \eqref{eq:c90}, yields
\begin{eqnarray*}
f&=& f_N+f_{\overline{N}}\\
&=& (f_N-\tau (f_N)^V)+(\tau (f_N)^V +f_{\overline{N}})\\
& =&(f_N-\tau (f_N)^V)+\sum_{\beta\in \mathrm{Sup}(f_{\overline{N}})} \frac{\tau}{T}
\left((f_N)^V-k_{\hat{\beta}} x^{\hat{\beta}} \right)+
\sum_{\beta\in \mathrm{Sup}(f_{\overline{N}})}  x^{\hat{\beta}}
\left (\frac{\tau k_{\hat{\beta}}}{T}+a_{\beta} x^{\delta}\right ).
\end{eqnarray*}
It is easy to see that
\begin{equation}\label{eq:c110}
\lim_{\parallel x \parallel \rightarrow 0 }
\left( \frac{\tau k_{\beta}}{T}+a_{\beta} x^{\delta} \right )= \frac{\tau k_{\beta}}{T}>0.
\end{equation}

This, together with \eqref{eq:c90} and \eqref{eq:c95}, implies that $f$ is local nonnegative.
Finally we proved the desired result.
\end{proof}
Two immediate consequences of Theorem \ref{th:2} are as follows.

\newtheorem{cor:1}{Corollary}
\begin{cor:1}\upshape
Let $f_N$ be the Newton's principal part of a polynomial $f\in \mathbb{R}[x_1,\cdots,x_n]$. If
 $f_N$ is homogenous and, moreover, positive definite ( i.e., $f_N$ is strictly positive
 over $\mathbb{R}^n\setminus O$ ), then $f$ is local nonnegative.
\end{cor:1}

\begin{corollary}\upshape\label{cor:c20}
Let $f_N$ be the Newton's principal part of a polynomial $f\in \mathbb{R}[x_1,\cdots,x_n]$,
and let $(f_N)^V$ is the characteristic polynomial of vertices of $f_N$.
If every $F$-face, $(f_N)_F$, of $f_N$ satisfies
\begin{equation}\label{eq:c120}
\forall x\in (\mathbb{R}\setminus 0)^n,\ (f_{N})_F> 0
\end{equation}
and, moreover, there is a term $x^{2d_i}_i\ (d_i>0)$ such that
\begin{equation} \label{eq:c130}
x^{2d_i}_i\in(f_N)^V, \ \forall x_i\in \mathbb{R}^n, i=1,\ldots,n.
\end{equation}
Then $O$ is an isolated singular point of the variety $V_{\mathbb{R}}(f)=\{x\in \mathbb{R}^n|\ f(x)=0 \}$.
\end{corollary}
For an example of Corollary \ref{cor:c20}, consider the polynomial in Example \ref{ex:1}.
$$f=x^2+y^4+z^6-sxy^2z^3$$
It is easy to see that for every $F$-face of $f_N$, \eqref{eq:c120}holds and so are \eqref{eq:c130}.
Then $O$ is an isolated singular point of the variety $V_{\mathbb{R}}(f)$.

\end{document}